\documentclass[a4paper,11pt]{amsart} 
\usepackage{amsmath,amsxtra,amssymb,latexsym, amscd,amsthm}
\usepackage[mathscr]{eucal}
\usepackage{mathrsfs}
\usepackage{bbm}
\usepackage{enumerate}
\usepackage{pict2e}
\usepackage{tikz}
\usepackage{graphicx}
\usepackage[a4paper]{geometry}
\usepackage{color}
\usepackage{hyperref}
\numberwithin{equation}{section}
\theoremstyle{plain}
\newtheorem{thm}{Theorem}[section]
\newtheorem{prp}[thm]{Proposition}

\newtheorem{lem}[thm]{Lemma}
\newtheorem*{euc*}{Euclidean division}
\newtheorem*{fek*}{Fekete's Lemma}
\newtheorem*{kin*}{Kingman's Subadditive Ergodic Theorem}
\newtheorem*{fur*}{Furstenberg-Kesten Theorem}
\theoremstyle{definition}

\newtheorem{rem}[thm]{Remark}

\newtheorem*{rem*}{Remark}

\newcommand{\dd}{\mathrm{d}}

 \newcommand{\Lim}{\mathop{\longrightarrow}\limits}
 \newcommand{\To}{\longrightarrow}
\renewcommand{\Im}{\operatorname{Im}}

\makeatletter
\newcommand*{\ov}[1]{%
  $\m@th\overline{\mbox{#1}}$%
}
\newcommand*{\ovA}[1]{%
  $\m@th\overline{\mbox{#1}\raisebox{3mm}{}}$%
}
\newcommand*{\ovB}[1]{%
  $\m@th\overline{\mbox{#1\rule{0pt}{3mm}}}$%
}
\newcommand*{\ovC}[1]{%
  $\m@th\overline{\mbox{#1\strut}}$%
}
\newcommand*{\ovD}[1]{%
  $\m@th\overline{\mbox{#1\vphantom{\"A}}}$%
}
\newcommand*{\ovE}[1]{%
  $\m@th\overline{\raisebox{0pt}[1.2\height]{#1}}$%
}
\newcommand*{\ovF}[1]{%
  $\m@th\overline{\raisebox{0pt}[\dimexpr\height+0.3mm\relax]{#1}}$%
}
\newcommand*{\ovG}[1]{%
  $\m@th\overline{\raisebox{0pt}[\dimexpr\height+1mm\relax]{#1\vphantom{A}}}$%
}
\makeatother
\newcommand{\N}{\mathbb{N}}
\newcommand{\Z}{\mathbb{Z}}

\newcommand{\R}{\mathbb{R}}
\newcommand{\C}{\mathbb{C}}

\newcommand{\eps}{\epsilon}
\newcommand{\cW}{\mathcal{W}}
\newcommand{\cP}{\mathcal{P}}
\newcommand{\cE}{\mathcal{E}}
\newcommand{\cK}{\mathcal{K}}
\newcommand{\cH}{\mathcal{H}}
\newcommand{\cG}{\mathcal{G}}
\newcommand{\cD}{\mathcal{D}}
\newcommand{\cA}{\mathcal{A}}
\newcommand{\tilg}{{\tilde{g}}}

\DeclareMathOperator{\expect}{\mathbf{E}}

\DeclareMathOperator{\prob}{\mathbf{P}}

\DeclareMathOperator{\supp}{supp}

\DeclareSymbolFont{extraup}{U}{zavm}{m}{n}
\DeclareMathSymbol{\varheart}{\mathalpha}{extraup}{86}
\DeclareMathSymbol{\vardiamond}{\mathalpha}{extraup}{87}

\title{Quantum ergodicity for the Anderson model on regular graphs}
\author{Nalini Anantharaman and Mostafa \textsc{Sabri}}
\address{Universit\'e de Strasbourg, CNRS, IRMA UMR 7501, F-67000 Strasbourg, France.}
\email{anantharaman@math.unistra.fr}
\address{Universit\'e de Strasbourg, CNRS, IRMA UMR 7501, F-67000 Strasbourg, France.}
\address{Department of Mathematics, Faculty of Science, Cairo University, Cairo 12613, Egypt.}
\email{sabri@math.unistra.fr}
\usepackage{calc}
\usepackage{yhmath}
\usepackage{graphicx}

\makeatletter
\newlength{\temp@wc@width}
\newlength{\temp@wc@height}
\newcommand{\widecheck}[1]{%
  \setlength{\temp@wc@width}{\widthof{$#1$}}%
  \setlength{\temp@wc@height}{\heightof{$#1$}}%
  #1\hspace{-\temp@wc@width}%
  \raisebox{\temp@wc@height+2pt}[\heightof{$\widehat{#1}$}]%
     {\rotatebox[origin=c]{180}{\vbox to 0pt{\hbox{$\widehat{\hphantom{#1}}$}}}}%
}
\makeatother

\begin{document}

\begin{abstract}
We prove a result of delocalization for the Anderson model on the regular tree (Bethe lattice). When the disorder is weak, it is known that large parts of the spectrum are a.s. purely absolutely continuous, and that the dynamical transport is ballistic. In this work, we prove that in such AC regime, the eigenfunctions are also delocalized in space, in the sense that if we consider a sequence of regular graphs converging to the regular tree, then the eigenfunctions become asymptotically uniformly distributed. The precise result is a quantum ergodicity theorem.
\end{abstract}

\maketitle

\section{Introduction}          \label{sec:introd}

\subsection{Background and discussion}
The Anderson model on an infinite graph $\mathbb{G}$ is a random Schr\"odinger operator $\cH^{\omega} = \mathcal{A}_{\mathbb{G}}+\cW^{\,\omega}$ which consists of the adjacency matrix with a random i.i.d. perturbation potential (throughout, i.i.d. stands for independent, identically distributed). Continuum analogs have also been studied, where the adjacency matrix is replaced by the Laplace operator on $\R^d$. Since the original paper of Anderson \cite{And} which discussed conduction properties in the presence of impurities, a large body of mathematical literature has been devoted to proving \emph{Anderson localization} under appropriate assumptions.

Localization in an interval $I\subseteq \R$ is mathematically interpreted in the following senses~:
\begin{itemize}
\item \emph{spectral localization :} for a.e. $\omega$, the spectrum of $\cH^{\omega}$ in $I$ is pure point,
\item \emph{exponential localization :} the corresponding eigenfunctions decay exponentially,
\item \emph{dynamical localization :} an initial state with energy in $I$ which is localized in a bounded domain essentially stays in this domain as time goes on.
\end{itemize}

These forms of localization have been proved for a wide variety of models, both on $\ell^2(\Z^d)$ and $L^2(\R^d)$. For the Anderson model on $\ell^2(\Z)$ or $L^2(\R)$, it is known \cite{KuS,CKM,DSS} that the full spectrum becomes localized for any nontrivial disorder. In higher dimensions, localization was proved under the conditions of high disorder or extreme energies in \cite{FS,DLS,DK,AM,Aiz,BK,GK}. \emph{Delocalization} is expected in the opposite regime of weak disorder, well inside the spectrum, in dimension $d\ge 3$. But proving this remains a challenging open problem. Here, delocalization is understood in the sense of~:
\begin{itemize}
\item \emph{spectral delocalization :} for a.e. $\omega$, the spectrum of $\cH^{\omega}$ in $I$ is purely absolutely continuous (AC for short),
\item \emph{spatial delocalization :} the corresponding (generalized) eigenfunctions do not concentrate on small regions. Ideally, they are uniformly distributed,
\item \emph{diffusive transport :} wave packets with energies in $I$ spread on the lattice at a rate $t^\alpha$ as time goes on (the transport is called ballistic when the rate is linear).
\end{itemize}

The Anderson model on $\ell^2(\mathbb{T}_q)$ is more approachable. Here $\mathbb{T}_q$ is the $(q+1)$-regular tree, $q\ge 2$. Localization at high disorder or energies beyond $[-(q+1),(q+1)]$ was proved in \cite{AM,Aiz}. The first mathematical results of delocalization were obtained in \cite{Klein,KleinB} (see also \cite{KuS2} for a previous sketch of ideas), where it was shown that at weak disorder, the transport is ballistic, and the spectrum is purely AC a.s. in closed subsets of $(-2\sqrt{q},2\sqrt{q})$. A new proof was found in \cite{FHS}, and it was later shown in \cite{AW,AW2} that spectral delocalization and ballistic transport hold in larger regions of the spectrum. Similar results were obtained for more general tree models in \cite{KLW,FHH}. One reason that makes trees technically simpler to analyze is the fact that the Green function on a tree satisfies some convenient recursion and factorization relations. In the physics literature, the self-consistent theory formulated in \cite{AAT,AT} becomes exact in the case of $\mathbb{T}_q$.

In view of the previous results, it is natural to ask whether \emph{spatial} delocalization also holds for the Anderson model on $\mathbb{T}_q$ at weak disorder. As the wavefunctions corresponding to AC spectrum are not square summable, one way to interpret spatial delocalization is to consider an orthonormal basis of eigenfunctions $(\psi_j)$ of the model on a sequence of finite graphs $(G_N)$ which converges to $\mathbb{T}_q$ in some sense, and show that $\psi_j$ becomes delocalized as $N\to \infty$. In contrast to the situation in $\Z^d$, one should not take $G_N$ to be the subtree $B_N\subset \mathbb{T}_q$ consisting of vertices at distance at most $N$ from a fixed root. The intuitive reason is that points on the boundary of $B_N$ satisfy different adjacency properties than the ones in $\mathbb{T}_q$. This is not a problem in $\Z^d$ because $\frac{|\partial B_N|}{|B_N|} \to 0$ as $N\to \infty$. On $\mathbb{T}_q$ however, we have $\frac{|\partial B_N|}{|B_N|} \to \frac{q-1}{q}>0$. In fact, the \emph{Benjamini-Schramm limit} of $B_N$ is not $\mathbb{T}_q$, but rather a \emph{canopy tree} \cite{Gab}, on which $\cH^{\omega}$ has only pure point spectrum at any disorder \cite{AW3}.

A better candidate for $(G_N)$ is a sequence of $(q+1)$-regular graphs with few short loops. In this case, the empirical spectral measures on $G_N$ converge weakly to the spectral measure of $\mathcal{A}_{\mathbb{T}_q}$. This is known as the \emph{law of Kesten-McKay} \cite{Kesten,MK}.

So consider a sequence of $(q+1)$-regular graphs $(G_N)$ with few short loops and let $N$ be large. The question of spatial delocalization turns out to be already nontrivial for the adjacency matrix alone (i.e. $H_N=\cA_{G_N}$ without potential), and it was only considered quite recently. In \cite{BL} it is shown that eigenfunctions of $\mathcal{A}_{G_N}$ cannot concentrate in small regions. A different criterion for delocalization was proved in \cite{ALM,A}, where it is shown that if $(\psi_j)$ is an orthonormal basis of eigenfunctions of $\mathcal{A}_{G_N}$, then the probability measure $\sum_{x\in G_N} |\psi_j(x)|^2\delta_x$ approaches the uniform measure $\frac{1}{N} \sum_{x\in G_N} \delta_x$ as $N\to \infty$. This shows that the mass of $\|\psi_j\|^2$ becomes evenly distributed on $G_N$ for large $N$. This is the type of results we prove in this paper for the Anderson model. It is known as \emph{quantum ergodicity}, in reference to its original framework on compact manifolds \cite{Shni,CdV,Zel}. We will discuss the result in more detail after stating our main theorems and compare it with predictions from the physics literature.

The results we just mentioned hold for \emph{deterministic} sequences of graphs $(G_N)$, and this is also the framework of our paper. In the context of \emph{random graphs}, it has recently been shown in \cite{BHY} that the eigenvectors of the adjacency matrix $H_N = \cA_{G_N}$ are completely delocalized. In particular, while quantum ergodicity is a result of delocalization \emph{for most} $\psi_j$, the authors in \cite{BHY} prove a probabilistic version of this \emph{for all} $\psi_j$ - this is known as \emph{quantum unique ergodicity}. Delocalization for the Anderson model $H_N^{\omega}=\mathcal{A}_{G_N}+W^{\omega_N}$ on random regular graphs has also been considered in \cite{Geisinger}. 

\subsection{Main results}

Consider the adjacency matrix $\mathcal{A}_{\mathbb{G}}$ of a graph $\mathbb{G}$,
\[
(\mathcal{A}_{\mathbb{G}}f)(v) = \sum_{w \sim v} f(w)
\]
for $f\in\ell^2(\mathbb{G})$ and $v\in\mathbb{G}$. Here $v \sim w$ means that $v$ and $w$ are nearest neighbors. We also denote by $\mathcal{N}_v$ the set of nearest neighbors of $v$.

Let $\mathbb{T}_q$ be the $(q+1)$-regular tree and fix an origin $o\in\mathbb{T}_q$. Consider the probability space $(\Omega,\mathbf{P})$, where $\Omega = \R^{\mathbb{T}_q}$ and $\mathbf{P} = \mathop \otimes_{v\in \mathbb{T}_q} \nu$. Here $\nu$ is a probability measure on $\R$. Given $\omega = (\omega_v)_{v \in \mathbb{T}_q}$, the Anderson model on $\mathbb{T}_q$ is the Schr\"odinger operator given by
\[
\cH_{\eps}^{\omega} = \mathcal{A}_{\mathbb{T}_q} + \cW^{\,\omega}_{\eps} \, , \qquad \text{where} \quad (\cW^{\,\omega}_{\eps}f)(v) := \eps\,\omega_v f(v) \, ,
\]
and $\eps \in \R$ parametrizes the disorder. As $\mathbf{P}$ is a product measure, the $\{\omega_v\}_{v \in \mathbb{T}_q}$ are i.i.d. with common distribution $\nu$. Expectation w.r.t. $\prob$ is denoted by $\expect$.

Let $(G_N)$ be a deterministic sequence of $(q+1)$-regular graphs with vertex set $V_N$, $|V_N|=N$. Then $\mathbb{T}_q$ is the universal cover of $G_N$ for all $N$. Let $\Omega_N = \R^{V_N}$ and $\cP_N = \mathop\otimes_{x\in V_N} \nu$ on $\Omega_N$. We denote $\widetilde{\Omega} = \prod_{N\in \N} \Omega_N$ and let $\mathcal{P}$ be any probability measure
\footnote{For example, one may take $\cP=\mathop\otimes \cP_N$, so that the $(\omega_N)$ for different values of $N$ are independent. Another interesting example is to restrict $\mathbf{P}$ to larger and larger boxes in $\mathbb{T}_q$ to define $\cP$. More precisely, let us write $G_N$ as a quotient $\Gamma_N\backslash\mathbb{T}_q$, where $\Gamma_N$ is a group of automorphisms of $\mathbb{T}_q$ acting without fixed points.
Let $\cD_N$ be a fundamental domain for the action of $\Gamma_N$ on the vertices of $\mathbb{T}_q$; it is in bijection with $V_N$ through the covering map $\mathbb{T}_q\To G_N$. Consider the map $\Phi: \Omega \To \widetilde{\Omega} = \prod_{N\in \N} \Omega_N$, which starting from $\omega\in \Omega$ defines $\omega_N$ as the restriction of $\omega$ to $\cD_N\simeq V_N$. If we define $\cP$ as the pushforward of $\mathbf{P}$ under $\Phi$, our results apply in this context.}
on $\widetilde{\Omega}$ having $\cP_N$ as marginal on $\Omega_N$. Given $(\omega_N)\in \widetilde{\Omega}$, so that $\omega_N=(\omega_x)_{x\in V_N}\in \Omega_N$, define
\[
H_N^{\omega} = \mathcal{A}_{G_N} + W^{\omega_N}_{\eps}\,, \quad \text{where} \quad (W^{\omega_N}_{\eps} \psi)(x) := \eps\,\omega_x \psi(x) \, .
\]

We make the following assumption on the potential~:

\medskip

\textbf{(POT)} The measure $\nu$ has a compact support, $\supp \nu \subseteq [-A,A]$, and is H\"older continuous, i.e. there exist $C_{\nu}>0$, $b\in (0,1]$ such that $\nu(I) \le C_{\nu} |I|^b$ for all bounded $I\subset \R$.

\medskip

The continuity of $\nu$ is only needed here to use some estimates from \cite{AW2}. If $\nu$ has a bounded density w.r.t. Lebesgue measure, this condition is satisfied with $b=1$.

It is known that $\sigma(\mathcal{A}_{\mathbb{T}_q}) = [-2\sqrt{q},2\sqrt{q}]$ and $\sigma(\cH_{\eps}^{\omega}) = \sigma(\mathcal{A}_{\mathbb{T}_q}) + \eps \supp \nu$, $\mathbf{P}$-a.s.

We next make the following assumptions on $(G_N)$~:

\medskip

\textbf{(EXP)} The sequence $(G_N)$ is a family of expanders: there exists $\beta>0$ such that the spectrum of $(q+1)^{-1}\mathcal{A}_{G_N}$ in $\ell^2(V_N)$ is contained in $[-1+\beta,1-\beta] \cup \{1\}$ for all $N$.

\medskip

\textbf{(BST)} For all $r>0$,
\[
\lim_{N \to \infty} \frac{|\{x \in V_N : \rho_{G_N}(x)<r\}|}{N} = 0 \, ,
\]
where $\rho_{G_N}(x)$ is the \emph{injectivity radius} at $x$, i.e. the largest $\rho$ such that the ball $B(x,\rho)$ in $G_N$ is a tree. 

\medskip

It is known that typical random $(q+1)$-regular graphs, and some explicit sequences of Ramanujan graphs, both satisfy \textbf{(EXP)} and \textbf{(BST)}. See \cite[Examples 1,2]{ALM} for details.

We may now state our main results~:

\begin{thm}                    \label{thm:1}
Assume that \emph{\textbf{(POT)}}, \emph{\textbf{(EXP)}} and \emph{\textbf{(BST)}} hold. Given $(\omega_N)\in \widetilde{\Omega}$, let $(\psi_i^{\omega_N})_{i=1}^N$ be an orthonormal basis of eigenfunctions of $H_N^{\omega}$ in $\ell^2(V_N)$, with corresponding eigenvalues $(\lambda_i^{\omega_N})_{i=1}^N$. Let $a_N:V_N \to \C$ be any function independent of $(\omega_N)$, such that $\sup_N\sup _{x \in V_N} |a_N(x)| \le 1$, and fix $0<\lambda_0<2\sqrt{q}$. There exists $\eps(\lambda_0)>0$ such that if $|\eps|<\eps(\lambda_0)$, we have for $\cP$-a.e. $(\omega_N)$,
\[
\lim_{N \to \infty} \frac{1}{N} \sum_{\lambda_i^{\omega_N}\in (-\lambda_0,\lambda_0)} \big|\langle \psi_i^{\omega_N},a_N \psi_i^{\omega_N} \rangle - \langle a_N \rangle \big|  = 0 \, ,
\]
where $\langle \psi_i^{\omega_N},a_N \psi_i^{\omega_N} \rangle = \sum_{x \in V_N} a_N(x) |\psi_i^{\omega_N}(x)|^2$ and $\langle a_N \rangle = \frac{1}{N} \sum_{x \in V_N} a_N(x)$.
\end{thm}

We next consider eigenfunction correlators. Here we assume the $(\psi_i^{\omega_N})_{i=1}^{N}$ are real-valued. More precisely, we need $\overline{\psi_i^{\omega_N}(x)}\psi_i^{\omega_N}(y)\in \R$ for all $i\le N$ and $x\sim y\in V_N$.

\begin{thm}                    \label{t:thm3}
Assume that \textbf{\emph{(POT)}}, \textbf{\emph{(EXP)}} and \textbf{\emph{(BST)}} hold. Given $(\omega_N)\in \widetilde{\Omega}$, let $(\psi_i^{\omega_N})_{i=1}^N$ be an orthonormal basis of eigenfunctions of $H_N^{\omega}$ in $\ell^2(V_N)$, with corresponding eigenvalues $(\lambda_i^{\omega_N})_{i=1}^N$. Assume the $(\psi_i^{\omega_N})_{i=1}^{N}$ are real-valued.

Let $K_N:V_N\times V_N \to \C$ be independent of $(\omega_N)$. Assume $\sup_N\sup_{x,y \in V_N} |K_N(x,y)| \le 1$ and $K_N(x,y)=0$ if $d(x,y)>R$. Fix $0<\lambda_0<2\sqrt{q}$. There exists $\eps(\lambda_0)>0$ such that if $|\eps|<\eps(\lambda_0)$, we have for $\mathcal{P}$-a.e. $(\omega_N)$,
\[
\lim_{\eta_0\downarrow 0} \lim_{N \to \infty} \frac{1}{N} \sum_{\lambda_i^{\omega_N}\in (-\lambda_0,\lambda_0)} \big|\langle \psi_i^{\omega_N},K_N \psi_i^{\omega_N} \rangle - \langle K_N \rangle_{\lambda_i^{\omega_N}}^{\eta_0} \big|  = 0 \, ,
\]
where
\begin{equation}\label{e:Klambda2}
\langle K \rangle_{\lambda}^{\eta_0} = \sum_{x,y \in V_N} K(x,y) \widetilde{\Phi}_{\lambda+i\eta_0}(\tilde{x},\tilde{y})  \quad \text{and} \quad \widetilde{\Phi}_{\gamma}(\tilde{x},\tilde{y}) = \frac{1}{N} \cdot \frac{\expect\left[\Im \mathcal{G}_{\eps}^{\gamma}(\tilde{x},\tilde{y})\right]}{\expect\left[\Im \mathcal{G}_{\eps}^{\gamma}(o,o)\right]} \, .
\end{equation}
\end{thm}

Here, $\tilde{x},\tilde{y}\in\mathbb{T}_q$ are lifts of $x,y\in V_N$ satisfying $d_{\mathbb{T}_q}(\tilde{x},\tilde{y})=d_{G_N}(x,y)$, and $\cG_{\eps}^{\gamma}(v,w) = \langle \delta_v, (\cH_{\eps}^{\omega}-\gamma)^{-1}\delta_w\rangle$ is the Green function of $\cH_{\eps}^{\omega}$.

Note that $\expect\left[ \Im \cG^{\gamma}_{\eps}(v,w) \right]=\expect\left[ \Im \cG^{\gamma}_{\eps}(o, u) \right]$ if $d(o, u)=d(v,w)$.

If $R=0$, we have $\langle a_N \rangle_{\lambda_i}^{\eta_0} = \frac{1}{N} \sum_{x\in V_N} a_N(x) = \langle a_N\rangle$.

\subsection{Consequences}
Denote $\psi_j = \psi_j^{\omega_N}$ and $\lambda_j=\lambda_j^{\omega_N}$. By a simple application of Markov's inequality, one may deduce from Theorem~\ref{thm:1} that for any $\varsigma>0$,
\[
\frac{1}{N} \# \left\{ \lambda_j \in (-\lambda_0,\lambda_0): |\langle \psi_j, a_N \psi_j\rangle - \langle a_N\rangle | > \varsigma \right\} \to 0
\]
$\cP$-a.s. as $N\to \infty$. So for most $\psi_j$ with eigenvalues in $(-\lambda_0,\lambda_0)$, the quantity $\langle \psi_j, a_N \psi_j\rangle$ approaches $\langle a_N\rangle$. If one takes $a_N = \delta_x$, this seems to imply that $|\psi_j(x)|^2 \approx \frac{1}{N}$ for any $x\in V_N$. However, one should pay attention to the speed of convergence. The best we can achieve is a negative power of the girth (see \cite{A}), which typically grows like $\log N$. Hence, if we wanted to apply the result to $a_N = \delta_x$, we would obtain $| |\psi_j(x)|^2 - \frac{1}{N}|^2 \lesssim \frac{1}{\log N}$, which does not imply that $|\psi_j(x)|^2 \approx \frac{1}{N}$. However, we may take $a_N$ to be the characteristic function of any $\Lambda_N \subset G_N$ of size $\alpha N$, with $0 < \alpha < N$ (e.g. $\alpha=\frac{1}{2}$). In this case, we obtain $| \|\chi_{\Lambda_N} \psi_j\|^2 - \alpha|^2 \lesssim \frac{1}{\log N}$, which implies that $\|\chi_{\Lambda_N} \psi_j\|^2 \approx \alpha$.

So Theorem~\ref{thm:1} implies that in weak disorder, most $\psi_j$ with eigenvalues in $(-\lambda_0,\lambda_0)$ are uniformly distributed in this sense~: if we consider any $\Lambda_N\subset G_N$ containing half the vertices of $G_N$, without any restriction on the shape of $\Lambda_N$, we find half the mass of $\|\psi_j\|^2$.

More generally, Theorem~\ref{t:thm3} implies that the correlation $\overline{\psi_j(x)}\psi_j(y)$ approaches the function $\widetilde{\Phi}_{\lambda_j+i0}(\tilde{x},\tilde{y})$ for large $N$.

A more detailed discussion of the consequences of our results is given in \cite{AS2}.

In the physics literature, a strong delocalization criterion is dubbed ``ergodicity''. In the ergodic phase, the inverse participation ratio $\|\psi_j\|_{2p}^{2p}$ should behave asymptotically like $\frac{1}{N^{p-1}}$. There seems to be a divergence among physicists concerning this issue for the Anderson model. Analytical predictions \cite{FM} first suggested that states in the delocalized phase are ergodic. Numerical evidence \cite{Altshuler1,Altshuler2} then put forward a different behavior of ``multi-fractality''. Finally, papers \cite{TMS,MC} supported the original predictions of ergodicity. In any case, our results do not allow us to settle this question. If we could zoom in at every $x$ and show that $|\psi_j(x)|^2 \approx \frac{1}{N}$, we would have $\|\psi_j\|_{2p}^{2p} \approx \frac{1}{N^{p-1}}$. But as mentioned above, we have to consider macroscopic regions $\Lambda_N$ instead.

\begin{rem}
In the main results, we assumed the ``test-observables'' $K_N$ are independent of $(\omega_N)$. This suffices to obtain the consequences discussed above. We can actually allow $K_N$ to depend on $(\omega_N)$, but in this case, we have to replace the average $\langle K_N\rangle_{\lambda}^{\eta_0}$ by the more complicated quantity \eqref{e:Klambda}. This remark is pertinent in particular for numerical tests \cite{Altshuler1,Altshuler2}, where one first picks a realization of the potential, then tests the ergodicity of eigenfunctions, so that the observable depends on $(\omega_N)$.
\end{rem}

\begin{rem}\label{rem:otherregions}
It was shown in \cite{AW} that AC spectrum exists outside $[-2\sqrt{q},2\sqrt{q}]$. It is natural to ask whether quantum ergodicity also holds in these regions. We do not answer this question here. Still, we may replace the interval $(-\lambda_0,\lambda_0)$ in our main results by any open set $I_1$ in which \cite[condition \textbf{(Green)}]{AS2} is true. For the Anderson model, this condition is the following.

Let $\hat{\zeta}_w^{\gamma}(v) = -\cG^{(v|w)}_{\eps}(v,v;\gamma)$ be the Green function on the subtree $\mathbb{T}_q^{(v|w)}$ in which the branch emanating from $v$, passing by $w\in\mathcal{N}_v$ is removed. Condition \textbf{(Green)} is said to hold on $I_1$ if
$\sup_{\lambda\in I_1,\eta\in(0,1)}\expect(|\Im \hat{\zeta}_o^{\lambda+i\eta}(o')|^{-s}) \le C_s<\infty$ for any $s>0$. Here, $o'$ is any neighbor of $o$. If \textbf{(POT)} holds, we show in Proposition~\ref{p:AW} that \textbf{(Green)} is equivalent to~:
\[
\inf_{\lambda\in I_1,\eta\in(0,1)} \expect\left[|\Im \hat{\zeta}_{o'}^{\lambda+i\eta}(o))|\right] \ge c >0 \quad \text{and} \quad \sup_{\lambda\in I_1,\eta\in(0,1)} \expect\left[(\Im \hat{\zeta}_{o'}^{\lambda+i\eta}(o))^2\right] \le C < \infty \,.
\]
The fact that \textbf{(Green)} implies both conditions is simple. The converse is nontrivial and uses important estimates from \cite{AW2}. The results of \cite{Klein} imply that both conditions hold true in proper subsets of $(-2\sqrt{q},2\sqrt{q})$ if the disorder is weak, but they may well be satisfied beyond this region.
\end{rem}

\section{Proof of the main results}

Our results follow from the main theorem in \cite{AS2}, which ensures quantum ergodicity once assumptions \textbf{(EXP)}, \textbf{(BSCT)} and \textbf{(Green)} are satisfied. We shall thus prove that the Anderson model satisfies \textbf{(BSCT)} for $\cP$-a.e. $(\omega_N)$ and that the local weak limit satisfies \textbf{(Green)}. We conclude by showing that the general average $\langle K \rangle_{\lambda+i\eta_0}$ appearing in \cite{AS2} can be replaced by the simpler average $\langle K \rangle_{\lambda}^{\eta_0}$ in the context of Theorem~\ref{t:thm3}.

We refer the reader to \cite[Section 1.6]{AS2} for a brief sketch of the strategy we follow to prove quantum ergodicity once this input has been checked.

A \emph{colored rooted graph} $(G,o,W)$ is a graph $G=(V,E)$ with a root $o$ and a map $W:V\to \R$, which we call ``coloring''. Let $\mathscr{G}_{\ast}^{D,A}$ be the set of (isomorphism classes of) colored rooted graphs $[G,o,W]$ with degree bounded by $D$ and coloring in $[-A,A]$. The set $\mathscr{G}_{\ast}^{D,A}$ is endowed with a topology, see for instance \cite[Appendix A]{AS2}. Checking \textbf{(BSCT)} means proving that, for any continuous $f:\mathscr{G}_{\ast}^{D,A}\to \R$, we have the convergence
\begin{equation}        \label{eq:lln}
\lim_{N\to \infty} \frac{1}{N} \sum_{x\in G_N} f([G_N,x,W^{\omega_N}]) =  \int_{\mathscr{G}_{\ast}^{D, \eps A}} f([G,v,W])\,\dd\, \mathbb{P}([G,v,W])\,,
\end{equation}
where $\mathbb{P}$ is some probability measure on $\mathscr{G}_{\ast}^{D,A}$. In Proposition~\ref{p:example} we prove that this holds for $\cP$-a.e. $(\omega_N)$, with $\mathbb{P}$ defined by \eqref{e:defP}. 
 
This says that the Anderson model on $\mathbb{T}_q$ is the \emph{Benjamini-Schramm limit} of the sequence $(G_N,W^{\omega_N})$ for $\cP$-a.e. $(\omega_N)$. By general arguments (e.g. \cite[Appendix A]{AS2}), this implies in particular that the empirical spectral measures of $(G_N,W^{\omega_N})$ converge $\cP$-almost surely to the integrated density of states of $\cH_{\eps}^{\omega}$ on $\mathbb{T}_q$.

\begin{prp} \label{p:example}
If $(G_N)$ satisfies \emph{\textbf{(BST)}}, then for $\cP$-a.e. $(\omega_N)\in\widetilde{\Omega}$, the sequence $(G_N,W^{\omega_N})$ has a local weak limit $\mathbb{P}$ which is concentrated on $\{[\mathbb{T}_q, o, \cW^{\,\omega}_{\eps}]:\omega\in\Omega\}$, where $o\in \mathbb{T}_q$ is fixed and arbitrary. The measure $\mathbb{P}$ acts by taking the expectation w.r.t. $\prob$, that is, if $D=q+1$, then
\begin{equation}\label{e:defP}
\int_{\mathscr{G}_{\ast}^{D, \eps A}} f([G,v,W])\,\dd\, \mathbb{P}([G,v,W]) = \int_{\Omega}f([\mathbb{T}_q,o,\cW^{\,\omega}_{\eps}])\,\dd\prob(\omega) = \expect\left[f([\mathbb{T}_q,o,\cW^{\,\omega}_{\eps}])\right]\,.
\end{equation}
\end{prp}
\begin{proof}
 Since $\eps$ is fixed and plays no role here, we omit it from the notation.
 
Consider the set $\mathscr{A}$ of continuous functions on $\mathscr{G}_{\ast}^{D,A}$, which ``depend only on a finite-size neighborhood of the origin''. That is, let $\mathscr{A} = \cup_{r\in \N} \mathscr{A}_r$, where
\begin{multline*}
\mathscr{A}_r = \big\{ f:\mathscr{G}_{\ast}^{D,A}\to \R : f \text{ is continuous and } f([G,v,W]) = f([G',v',W']) \\
 \text{ if } [B_G(v,r),v,W]=[B_{G'}(v',r),v',W'] \big\} \, .
\end{multline*}
Then $\mathscr{A}$ is an algebra of continuous functions containing $1$ which separates points.

Using the compactness of $\mathscr{G}_{\ast}^{D,A}$, cf. \cite[Section A.1]{AS2}, it suffices to show that there exists $\Omega_0 \subseteq \widetilde{\Omega}$ with $\mathcal{P}(\Omega_0)=1$ such that for any $(\omega_N) \in \Omega_0$ and any $f\in \mathscr{A}$ the limit in \eqref{eq:lln} holds true.
For this, we adapt the strong law of large numbers for random variables in $L^4$ (see e.g. \cite{Durr}). Denote by $\mathcal{E}$ the expectation w.r.t. $\mathcal{P}$. Given $f\in \mathscr{A}_r$, let
\[
Y_x=Y_x^N = f([G_N,x,W^{\omega_{N}}]) - \mathcal{E}\left[f([G_N,x,W^{\omega_{N}}])\right] \quad \text{and} \quad S_N = \frac{1}{N}\sum_{x\in V_N} Y_x \, .
\]
Then $Y_x^N$ only depends on $(\omega_z)_{z\in B_{G_N}(x,r)}$, since $f([G_N,x,W^{\omega_N}]) = f([G_N,x,\tilde{W}^{\omega_N}])$ if $W^{\omega_N} = \tilde{W}^{\omega_N}$ on $B_{G_N}(x,r)$. Hence, $Y_x^N$ and $Y_y^N$ are independent if $d_{G_N}(x,y)>2r$. Moreover, each $Y_x^N$ is bounded by $2\|f\|_\infty$. Now
\begin{multline*}
\cE\bigg[\sum_{x\in V_N} Y_x\bigg]^4 =  \sum_{x\in V_N} \cE(Y_x^4) +6\sum_{x,y\in V_N}\cE(Y_x^2 Y_y^2) +4\sum_{x,y\in V_N}\cE(Y_x Y_y^3+ Y_y Y_x^3)\\+12\,
\sum_{x,y,z\in V_N}\cE(Y_x Y_y Y_z^2+ Y_x Y^2_y Y_z+Y^2_x Y_y Y_z ) + 24\sum_{x,y,z,t\in V_N } \cE(Y_x Y_y Y_z Y_t) \,.
\end{multline*}
The terms involving $\sum_{x\in V_N}$ and $\sum_{x,y\in V_N}$ are respectively $O(N)$ and $O(N^2)$. In the other terms, some cancellations take place due to independence. Indeed, $\cE(Y_x Y_y Y_z^2)$ vanishes as soon as $ d_{G_N}(x,y)>4r$~: in that case, we have either $d_{G_N}(x, z)>2r$ and $Y_x$ is independent of the pair $(Y_y, Y_z)$, or $d_{G_N}(y, z)>2r$ and $Y_y$ is independent of $(Y_x, Y_z)$. Thus we have either $\cE(Y_x Y_y Y_z^2)=\cE(Y_x)\cE( Y_y Y_z^2)=0$ or $\cE(Y_y Y_x Y_z^2)=\cE(Y_y)\cE( Y_x Y_z^2)=0$ .
Hence, $|\sum_{x, y, z\in V_N}\cE(Y_x Y_y Y_z^2)|\leq N^2 \tau_{q,4r}(2\|f\|_\infty)^4$, where $\tau_{q,r}:=|B_{\mathbb{T}_q}(o,r)|$.

Similarly, for $\cE( Y_x Y_y Y_z Y_t)$ to be non zero, each point must be at distance $\leq 2r$ from one of the three others. So we must have [$ d_{G_N}(x,y)\leq 2r$ and $ d_{G_N}(z,t)\leq 2r$] (or a permutation thereof) or [$ d_{G_N}(x, \bullet)\leq 8r$ for $\bullet=y, z, t$]. Hence, $\sum_{x, y, z, t\in V_N }| \cE(Y_x Y_y Y_z Y_t)|\leq 3N^2 \tau_{q,2r}^2(2\|f\|_\infty)^4 + N \tau_{q,8r}^3(2\|f\|_\infty)^4 $. We thus get
$\mathcal{E}(|S_N|^4)\leq C(r, f) N^{-2}$ in all cases, with $r$ fixed.

Using Borel-Cantelli, it follows as in \cite[Theorem 2.3.5]{Durr} that $S_N\to 0$ $\cP$-a.s., i.e. for $(\omega_N)\in \Omega_f$ with $\cP(\Omega_f)=1$. Since $\mathscr{A}$ has a countable dense subset $\{f_j\}$ (e.g. the functions taking only rational values), we let $\Omega_0=\cap_j \Omega_{f_j}$ and get $\cP(\Omega_0)=1$.

Now if $f\in \mathscr{A}$, say $f\in \mathscr{A}_r$, we have
\begin{align}    \label{eq:replace}
& \Big| \frac{1}{N}\sum_{x\in V_N} f([G_N,x,W^{\omega_N}]) - \expect\left[f([\mathbb{T}_q,o,\cW^{\,\omega}])\right]\Big| \nonumber \\
& \qquad \le |S_N| + \Big|\frac{1}{N}\sum_{x\in V_N} \mathcal{E}_N[f([G_N,x,W^{\omega_N}])] - \expect\left[f([\mathbb{T}_q,o,\cW^{\,\omega}])\right]\Big| \, .
\end{align}
If $\rho_{G_N}(x)\ge r$, there is a graph isomorphism $\phi_x:B_{G_N}(x,r) \to B_{\mathbb{T}_q}(o,r)$ with $\phi_x(x)=o$. Denoting $\cW_x^{\,\omega} = W^{\omega_N}\circ\phi_x^{-1}$, we get $[B_{G_N}(x,r),x,W^{\omega_N}] = [B_{\mathbb{T}_q}(o,r),o,\cW_x^{\,\omega}]$, so $f([G_N,x,W^{\omega_N}]) = f([\mathbb{T}_q,o,\cW_x^{\,\omega}])$. Using standard measure-preserving transformations, it follows that $\mathcal{E}_N[f([G_N,x,W^{\omega_N}])] = \expect\left[f([\mathbb{T}_q,o,\cW^{\,\omega}])\right]$. Hence,
\[
\Big| \frac{1}{N}\sum_{x\in G_N} f([G_N,x,W^{\omega_N}]) - \expect[f([\mathbb{T}_q,o,\cW^{\,\omega}])]\Big| \le |S_N| + \frac{\# \{x:\rho_{G_N}(x)<r\}}{N}(2\|f\|_{\infty}) \, .
\]
Taking $N\to \infty$, it follows by \textbf{(BST)} that if $(\omega_N)\in \Omega_0$, then (\ref{eq:lln}) is true for any $f\in \{f_j\}$, the dense subset of $\mathscr{A}$. Arguing as in \cite[Corollary 15.3]{Klenke}, the proof is complete.
\end{proof}

We next show that the measure $\mathbb{P}$ in Proposition~\ref{p:example} satisfies assumption \textbf{(Green)} in $I_1=(-\lambda_0,\lambda_0)$, i.e. $\sup_{\lambda\in I_1,\eta\in (0,1)}\mathbb{E}(\sum_{o'\sim o} |\Im \hat{\zeta}_o^{\lambda+i\eta}(o')|^{-s})<\infty$ for any $s>0$, where $\hat{\zeta}_{o}^{\gamma}(o')=-\cG^{(o'|o)}_{\eps}(o',o';\gamma)$. Since $\mathbf{E}(|\Im\hat{\zeta}_{o}^{\gamma}(o')|^{-s})$ is the same for any $o'\sim o$ and also equal to $\mathbf{E}(|\Im \hat{\zeta}_{o'}^{\gamma}(o)|^{-s})$, the measure $\mathbb{P}$ satisfies \textbf{(Green)} in $I_1$ iff for any $s>0$, we have $\sup_{\lambda\in I_1,\eta\in(0,1)}\expect(|\Im \hat{\zeta}_o^{\lambda+i\eta}(o')|^{-s})<\infty$, where $o'\sim o$ is arbitrary. To check this condition, we use the moment bounds provided in \cite{Klein} and \cite{AW2}.

\begin{prp}  \label{p:AW}
(i) Under assumption \emph{\textbf{(POT)}}, condition \emph{\textbf{(Green)}} holds in $I_1\subseteq\R$ iff
\begin{equation}\label{e:greenequiv}
\inf_{\lambda\in I_1,\eta\in(0,1)} \expect\left[|\Im \hat{\zeta}_{o'}^{\lambda+i\eta}(o))|\right] \ge c >0 \quad \text{and} \quad \sup_{\lambda\in I_1,\eta\in(0,1)} \expect\left[(\Im \hat{\zeta}_{o'}^{\lambda+i\eta}(o))^2\right] \le C < \infty \,.
\end{equation}

(ii) Assuming \emph{\textbf{(POT)}} holds and $0<\lambda_0<2\sqrt{q}$, there exists $\eps(\lambda_0)>0$ such that if $|\eps|<\eps(\lambda_0)$, then assumption \emph{\textbf{(Green)}} holds on $I_1= (-\lambda_0,\lambda_0)$.
\end{prp} 
\begin{proof}
The fact that \textbf{(Green)} implies the first condition follows from Jensen's inequality~: $\expect[|\Im \hat{\zeta}_{o'}^{\lambda+i\eta}(o)|^{-1}] \ge \expect[|\Im \zeta_{o'}^{\lambda+i\eta}(o)|]^{-1}$. To see that \textbf{(Green)} implies the second one, recall that $|\hat{\zeta}_{o'}^{\gamma}(o)|^2=|\eps\,\omega_o-\gamma+\sum_{o''\in \mathcal{N}_o\setminus \{o'\}} \hat{\zeta}_{o}^{\gamma}(o'')|^{-2} \le |\Im \hat{\zeta}_o^{\gamma}(o'')|^{-2}$ by classical recursive formulas (see e.g. \cite[Lemma 2.1]{AS}), so the second property follows.

Conversely, given $\delta\in (0,1)$, introduce the set
\[
\sigma_{ac}^{\eps}(\delta)=\{\lambda\in \R: \mathbf{P}(|\Im \hat{\zeta}_{o'}^{\lambda+i\eta}(o)|>\delta)>\delta \quad \forall \eta\in (0,1)\}\,,
\]
where $o'\sim o$ is arbitrary. The arguments of \cite[Theorem 2.4]{AW2} show that if \textbf{(POT)} holds, then assumption \textbf{(Green)} is satisfied on any bounded $I\subset \sigma_{ac}^{\eps}(\delta)$. We now show that if the two conditions in \eqref{e:greenequiv} hold, then $I_1\subseteq \sigma_{ac}^{\eps}(c')$ for some $c'>0$.

Suppose that $\prob(|\Im \hat{\zeta}_{o'}^{\lambda+i\eta}(o)| > c') \le c'$ for some $\lambda\in I_1$ and $\eta\in (0,1)$. Then
\begin{align*}
\expect\left[|\Im \hat{\zeta}_{o'}^{\lambda+i\eta}(o)|\right] & = \expect\left[|\Im \hat{\zeta}_{o'}^{\lambda+i\eta}(o)|1_{|\Im \zeta_{o'}^{\lambda+i\eta}(o)|>c}\right] \\
& \quad + \expect\left[|\Im \hat{\zeta}_{o'}^{\lambda+i\eta}(o)|1_{|\Im \hat{\zeta}_{o'}^{\lambda+i\eta}(o)|\le c}\right] \le (Cc')^{1/2} + c'< c
\end{align*}
if $c'>0$ is small enough, yielding a contradiction. Here, we used Cauchy-Schwarz and the bound $\expect[|\Im \hat{\zeta}_{o'}^{\lambda+i\eta}(o)|^2] \le C$. Hence, $I_1\subseteq \sigma_{ac}^{\eps}(c')$ for some $c'>0$ as claimed, proving (i)

To prove (ii), we show that if $I_1=(-\lambda_0,\lambda_0)$, then the conditions in \eqref{e:greenequiv} are satisfied if $|\eps|$ is small. This is a consequence of the results of \cite[Theorem 1.3]{Klein}, according to which there exists $\eps(\lambda_0)$ such that if $|\eps|<\eps(\lambda_0)$, then the second condition is satisfied. For the first one, recall that if $\eps = 0$, then $\hat{\zeta}^{\lambda+i0}_{o'}(o) = \frac{\lambda-i\sqrt{4q-\lambda^2}}{2q}$ for any $\lambda\in(-2\sqrt{q},2\sqrt{q})$. In particular, $|\Im \hat{\zeta}_{o'}^{\lambda+i0}(o)| > 2c_{\lambda_0}>0$ for all $\lambda\in [-\lambda_0,\lambda_0]$. By \cite[Theorem 1.4]{Klein}, there exists $\epsilon_0=\epsilon(\lambda_0)$ such that the map $(\eps,\lambda,\eta)\mapsto \expect[|\Im \hat{\zeta}_{o'}^{\lambda+i\eta}(o)|]$ has a continuous extension\footnote{In \cite{Klein}, this is proved for the whole Green function $\cG^{\gamma}_{\eps}(o,o)$. To see this for $\hat{\zeta}_{o'}^{\gamma}(o)$, just replace the $(q+1)$-regular tree by the rooted $q$-ary tree. More precisely, represent $\expect[\hat{\zeta}_{o'}^{\gamma}(o)] = \frac{-i}{\pi} \int_{\R^2} (B_{\eps,\gamma} \zeta_{\eps,\gamma}^q)(\varphi^2)d^2\varphi$ instead of $\expect[\cG^{\gamma}_{\eps} (o,o)] = \frac{i}{\pi} \int_{\R^2} (B_{\eps,\gamma} \zeta_{\eps,\gamma}^{q+1})(\varphi^2)d^2\varphi$, where $B_{\eps,\gamma}$ and $\zeta_{\eps,\gamma}$ are defined in \cite{Klein}. Then both continuity assertions follow from \cite[Theorem 3.5]{Klein} and the discussion thereafter.} to $(-\eps_0,\eps_0)\times [-\lambda_0,\lambda_0]\times [0,\infty)$. Hence, we may find $\epsilon_1\le \epsilon_0$ and $\eta_1>0$ such that $\expect[|\Im \hat{\zeta}_{o'}^{\lambda+i\eta}(o)|] > c_{\lambda_0}$ for all $(\eps,\lambda,\eta)\in (-\eps_1,\eps_1)\times[-\lambda_0,\lambda_0]\times [0,\eta_1)$. So the first condition in \eqref{e:greenequiv} holds if the $\inf$ on $\eta$ is taken over $(0,\eta_1)$. To extend this to $(0,1)$, note that $\Im \hat{\zeta}_{o'}^{\lambda+i\eta}(o) = \eta \| (\cH_{\eps}^{(o|o')}-\lambda-i\eta)^{-1}\delta_o\|^2$. Since $\|\cH_{\eps}^{(o|o')}-\lambda-i\eta\|_{\ell^2\to\ell^2} \le (q+1) + \eps A+c_{I_1}+1=:\tilde{c}$ for any $\lambda\in I_1$ and $\eta\in (0,1)$, we have the deterministic bound $|\Im \hat{\zeta}_{o'}^{\lambda+i\eta}(o)|\ge \eta \tilde{c}^{-2}$. This completes the proof.
\end{proof}

Denote $G_N=(V_N,E_N)$, $I_1=(-\lambda_0,\lambda_0)$, and let $(\psi_j^{\omega_N})$ be an orthonormal basis of $\ell^2(V_N)$ with corresponding eigenvalues $(\lambda_j^{\omega_N})$. Using Propositions~\ref{p:example} and \ref{p:AW} (ii), it follows that \cite[Theorems 1.1,1.3]{AS2} hold $\cP$-a.s. on $I_1$ in the regime of weak disorder $|\eps|<\eps(\lambda_0)$. We thus get the following statement~:

\emph{Fix $R\in \N$. Let $K=K_N:V_N \times V_N\to \C$ satisfy $K(x,y)=0$ if $d(x,y)>R$ and $\sup_N \sup_{x,y\in V_N}|K_N(x,y)|\le 1$. Given $\gamma\in \C \setminus \R$, define the weighted average}
\begin{equation}\label{e:Klambda}
\langle K\rangle_{\lambda+i\eta_0}=  \sum_{x,y\in V_N}K(x, y) \Phi_{\lambda+i\eta_0}^N(\tilde{x},\tilde{y}) \quad \text{where} \quad \Phi_{\gamma}^N(\tilde{x},\tilde{y}) = \frac{\Im {\tilde g^{\gamma}_N}(\tilde x, \tilde y)}{\sum_{x\in V_N} \Im {\tilde g^{\gamma}_N}(\tilde x,\tilde x)} \,.
\end{equation}
\emph{Then there is $\Omega_0\subseteq \widetilde{\Omega}$, $\cP(\Omega_0)=1$, and $\eps(\lambda_0)>0$, such that for any $(\omega_N)\in \Omega_0$ and $|\eps|<\eps(\lambda_0)$, we have}
\[
\lim_{\eta_0\downarrow 0} \lim_{N\to \infty} \sum_{\lambda_j^{\omega_N}\in I_1} \left| \langle \psi_j^{\omega_N},K\psi_j^{\omega_N}\rangle - \langle K\rangle_{\lambda_j^{\omega_N}+i\eta_0}\right| = 0\,.
\]
Here, $\tilde{g}^{\gamma}_N(\tilde{x},\tilde{x})$ is the Green function of the ``lifted'' operator given by $\widetilde{H}_N^{\omega} = \mathcal{A}_{\mathbb{T}_q} + \widetilde{W}^{\omega_N}_{\eps}$, where $\widetilde{W}^{\omega_N}_{\eps}(v) = W^{\omega_N}_{\eps}(\pi_N v)$ and $\pi_N:\mathbb{T}_q\to G_N$ is the covering projection.

To conclude the proof of Theorems~\ref{thm:1}, \ref{t:thm3}, we show that if $K$ is independent of $(\omega_N)$, then we may replace $\langle K\rangle_{\lambda+i\eta_0}$ by the simpler average $\langle K \rangle_{\lambda}^{\eta_0}$.
 
\begin{lem}             \label{lem:newav}
Let $K_N: V_N \times V_N \to \C$ be independent of $(\omega_N)$ and satisfy $K_N(x,y) = 0$ if $d(x,y)>R$ and $\sup|K_N(x,y)| \le 1$. Fix $\eta_0>0$, and define $\langle K_N\rangle_{\lambda+i\eta_0}$ and $\langle K_N\rangle_{\lambda}^{\eta_0}$ as in \emph{(\ref{e:Klambda})} and \emph{(\ref{e:Klambda2})}. Then for $\mathcal{P}$-a.e. $(\omega_N)$,
\[
\left|\langle K_N \rangle_{\lambda +i\eta_0}- \langle K_N \rangle_{\lambda}^{\eta_0}\right| \Lim_{N\To +\infty} 0
\]
uniformly for $\lambda$ in a compact set.
\end{lem}
\begin{proof}
By the triangle inequality, we may assume $K_N(x,y) \neq 0$ only if $d(x,y)=R$.

Let $I\subset \R$ be compact, and let $\lambda\in I$.
Denote $\gamma =\lambda+i\eta_0$ and $S_R^x = \{y : d(x,y)=R\}$.

By Proposition~\ref{p:example}, there is $\Omega_0\subseteq\widetilde{\Omega}$ with $\cP(\Omega_0)=1$ such that for any $(\omega_N)\in \Omega_0$, $(G_N,W^{\omega_N})$ has a local weak limit $\mathbb{P}$. As in \cite[equation A.13]{AS2}, we now see $K_N$ as an additional coloring. Up to passing to a subsequence, for any $(\omega_N)\in \Omega_0$, $(G_N,W^{\omega_N},K_N)$ has a local weak limit $\hat{\mathbb{P}}$ whose marginal on $\mathscr{G}_{\ast}^{D,A}$ coincides with $\mathbb{P}$. So by \cite[equation A.13]{AS2}, we have for any $(\omega_N)\in \Omega_0$ and uniformly in $\lambda\in I$, 
\begin{equation}\label{e:firslimit}
\langle K_N \rangle_{\lambda+i\eta_0} \Lim_{N\To \infty} \frac{\hat{ \mathbb{E}}\left(\sum_{w\in S_R^v}\cK(v, w)\Im \cG^\gamma_{\eps}(v, w)\right)}{\mathbb{E}\left[\Im \cG_{\eps}^{\gamma}(v,v)\right]} \, ,
\end{equation}
where $\cK:V(\mathbb{T}_q)\times V(\mathbb{T}_q)\to \C$ is a random coloring on $\mathbb{T}_q$ with values in $\{|z|\le 1\}$. By Proposition~\ref{p:example}, $\mathbb{E}[\Im \cG_{\eps}^{\gamma}(v,v)] = \mathbf{E}[\Im \cG_{\eps}^{\gamma}(o,o)]$. On the other hand,
\begin{equation}     \label{eq:klimit}
\frac{1}{N} \sum_{x\in V_N} \sum_{y\in S_R^x}  K_N(x,y) \Lim_{N\To +\infty}\hat{\mathbb{E}}\left(\sum_{w\in S_R^v}\cK(v, w)\right)\,.
\end{equation}
Since $\langle K_N\rangle_{\lambda}^{\eta_0} = \frac{\expect[\Im \cG_{\eps}^{\gamma}(o,y_R)]}{\expect[\Im \cG_{\eps}^{\gamma}(o,o)]} \cdot \frac{1}{N} \sum_{x\in V_N}  \sum_{y\in S_R^x}  K_N(x,y)$, where $y_R$ is any point at distance $R$ from $o$, we thus have
\begin{equation}\label{e:secondlimit}
\langle K_N\rangle_{\lambda}^{\eta_0} \Lim_{N\To +\infty} \frac{\hat{\mathbb{E}}\left(\sum_{w\in S_R^v}\cK(v, w)\right)  \mathbf{E}\left[\Im \cG^\gamma_{\eps}(o, y_R)\right]}{\mathbf{E}[\Im \cG_{\eps}^{\gamma}(o,o)]}
\end{equation}
uniformly in $\lambda\in I$. Comparing \eqref{e:firslimit}, \eqref{e:secondlimit}, the lemma will follow if we show that the limits are equal, i.e. for any fixed $\gamma=\lambda+i\eta_0$, we have
\begin{equation}\label{eq:indep}
\hat{ \mathbb{E}}\left(\sum_{w\in S_R^v}\cK(v, w)\Im \cG^\gamma_{\eps}(v, w)\right) = \hat{\mathbb{E}}\left(\sum_{w\in S_R^v}\cK(v, w)\right)  \mathbf{E}\left[\Im \cG^\gamma_{\eps}(o, y_R)\right]\,.
\end{equation}
In fact, we know the LHS is the limit of $\frac{1}{N}\sum_{x\in V_N,y\in S_R^x} K_N(x,y)\Im \tilg_N^{\gamma}(\tilde{x},\tilde{y})$ for any $(\omega_N)\in\Omega_0$, so it suffices to show that $\frac{1}{N}\sum_{x\in V_N,y\in S_R^x} K_N(x,y)\Im \tilg_N^{\gamma}(\tilde{x},\tilde{y})$ also converges to the RHS. For this, let $h_{\eta_0}(t)=(t-i\eta_0)^{-1}$ and given $\delta>0$, let $Q_{\delta}$ be a polynomial with $\|h_{\eta_0}-Q_{\delta}\|_{\infty}<\delta$. Then for any $v,w\in\mathbb{T}_q$,
\begin{equation}\label{eq:greenpoly}
|\tilg^{\gamma}_N(v,w)-Q_{\delta}(\widetilde{H}_N-\lambda)(v,w)|<\delta \quad \text{and} \quad |\cG^{\gamma}_{\eps}(v,w)-Q_{\delta}(\cH_{\eps}^{\omega}-\lambda)(v,w)|<\delta\,.
\end{equation}
Now consider
\[
Y_x = \sum_{y\in S_R^x}K_N(x,y)(\Im Q_{\delta}(\widetilde{H}_N-\lambda)(\tilde{x},\tilde{y}) - \cE[\Im Q_{\delta}(\widetilde{H}_N-\lambda)(\tilde{x},\tilde{y})])
\]
and $S_N = \frac{1}{N}\sum_{x\in V_N}Y_x$. If $Q_{\delta}$ has degree $d_{\delta}$, then $Q_{\delta}(\widetilde{H}_N-\lambda)(\tilde{x},\tilde{y})$ only depends on $(\omega_z)_{z\in B_{G_N}(x,d_{\delta})}$. So $Y_x$ and $Y_y$ are independent if $d_{G_N}(x,y)>2d_{\delta}$. Repeating the arguments of Proposition~\ref{p:example}, we see that $S_N\to 0$ on a set $\Omega_{\lambda}$ of full probability. We may thus replace $\Im Q_{\delta}(\widetilde{H}_N-\lambda)(\tilde{x},\tilde{y})$ by $\cE[\Im Q_{\delta}(\widetilde{H}_N-\lambda)(\tilde{x},\tilde{y})]$ as in \eqref{eq:replace}.

We have again $\cE[\Im Q_{\delta}(\widetilde{H}_N-\lambda)(\tilde{x},\tilde{y})] = \mathbf{E}[\Im Q_{\delta}(\cH_{\eps}^{\omega}-\lambda)(o,y_R)]$ if $\rho_{G_N}(x)\ge d_{\delta}+R$. So using \eqref{eq:greenpoly}, \eqref{eq:klimit} and \textbf{(BST)}, we get for any $(\omega_N)\in \Omega_{\lambda} \cap \Omega_0$,
\begin{multline*}
\Big| \frac{1}{N}\sum_{x\in V_N,y\in S_R^x} K_N(x,y)\Im \tilg_N^{\gamma}(\tilde{x},\tilde{y}) -  \hat{\mathbb{E}}\left(\sum_{w\in S_R^v}\cK(v, w)\right)  \mathbf{E}\left[\Im \cG^\gamma_{\eps}(o, y_R)\right] \Big| \\
\le |S_N| + \frac{\#\{x:\rho_{G_N}(x)<d_{\delta}+R\}}{N}(2(q+1)q^{R-1}\eta_0^{-1}) + 2(q+1)q^{R-1}\delta \To 0
\end{multline*}
as $N\to\infty$ followed by $\delta\downarrow 0$. So for $(\omega_N)\in \Omega_{\lambda}\cap \Omega_0$, $\frac{1}{N}\sum_{x\in V_N,y\in S_R^x} K_N(x,y)\Im \tilg_N^{\gamma}(\tilde{x},\tilde{y})$ converges to both the LHS and RHS of \eqref{eq:indep}, and the claim follows.
\end{proof}

\medskip

{\bf{Acknowledgements~:}} This material is based upon work supported by the Agence Nationale de la Recherche under grant No.ANR-13-BS01-0007-01, by the Labex IRMIA and the Institute of Advance Study of Universit\'e de Strasbourg, and by Institut Universitaire de France.

\providecommand{\bysame}{\leavevmode\hbox to3em{\hrulefill}\thinspace}
\providecommand{\MR}{\relax\ifhmode\unskip\space\fi MR }
\providecommand{\MRhref}[2]{%
  \href{http://www.ams.org/mathscinet-getitem?mr=#1}{#2}
}
\providecommand{\href}[2]{#2}

\end{document}